\documentclass[11pt,letterpaper]{amsart}%
\usepackage{amsmath}
\usepackage{amsfonts}
\usepackage{amssymb}
\usepackage{graphicx}%
\providecommand{\U}[1]{\protect\rule{.1in}{.1in}}
\newtheorem{theorem}{Theorem}

\newtheorem{proposition}[theorem]{Proposition}
\newtheorem{remark}[theorem]{Remark}

\numberwithin{equation}{section}
\begin{document}

\title[stability of optimal transport regularity ]{A quantitative stability result for regularity of optimal
transport on compact manifolds}
\author{Micah Warren}
\address{Department of Mathematics\\
University of Oregon, Eugene, OR 97403}
\email{micahw@uoregon.edu}

\begin{abstract}
We use a Korevaar-style maximum principle approach to show the following:
Fixing a $C^{2}$ bound on the log densities of a set of smooth measures, there is a
quantifiably-sized Wasserstein neighborhood over which all pairs of such measures will enjoy
smooth optimal transport. \ We do this in spite of unhelpful
MTW\ curvature, by showing that when the gradient of the Kantorovich potential is small enough, the Hessian ``bound"
places the Hessian in one of two disconnected regions, one bounded and the
other unbounded. \ Tracking  the estimate along a continuity path which starts
in the bounded region, we conclude the Hessian must stay bounded.

\end{abstract}

\maketitle
\section{Introduction}
Two decades ago, Ma, Trudinger and Wang \cite{MTW} proved the marquee result in the
regularity theory for optimal transport equations; namely that there is a 4th-order condition on the cost function which is essential to proving interior
smoothness of optimal transportation maps. \ Under some basic 
conditions, the positivity of the Ma-Trudinger-Wang curvature tensor is sufficient for regularity; that is, smooth, positive measures will be paired by regular
optimal transport maps. \ Shortly after, Loeper \cite{Loeper} showed that if
this quantity is negative anywhere, one can find smooth measures whose optimal
pairing is discontinuous. 

In practice, however, the MTW\ condition remains difficult to come by. \ If
the cost is distance-squared on a Riemannian manifold, the condition immediately fails on a manifold
with negative sectional curvature  \cite{Loeper}, and is difficult or impossible to prove when the curvature is
positive but non-homogeneous. While the condition holds on the round sphere \cite{DL2006}, it requires some effort to show even on high-regularity perturbations of spheres, (see \cite{DG}, \cite{LV1}, \cite{FR}, \cite{FRV}, \cite{GeYe}.)   

In the absense of a full regularity theory encompassing all smooth measures, one can study partial regularity (as done in \cite{MR3349831}) or one may ask if there
are reasonable sufficient conditions that may involve both the cost function
and the measures themselves, under which one can prove regularity. \ Here, we
offer an answer in the latter direction; if the smooth measures on a compact manifold
share a $C^{2}$ bound on the log densities, their pairing will be smooth
provided they are close (depending on the $C^2$ bound)\ to one another in
Wasserstein space.

In particular,

\begin{theorem}\label{main}
Suppose $\left(  M,g_{0}\right)  $ is a compact Riemannian manifold. \ \ Given
any $\bar{C}< \infty $, there is a $\delta(\bar{C},M,g_{0})>0 $ such that if $\mu$ and
$v$ are smooth measures satisfying
\[
\left\Vert \log\frac{d\mu}{dVol}\right\Vert _{C^{2}},\left\Vert \log\frac
{d\nu}{dVol}\right\Vert _{C^{2}}\leq\bar{C}%
\]
and
\[
W_2^{2}(\mu,\nu)<\delta
\]
then there is a $C_{0}(\bar{C})$ such that the optimal transport map $T$
between $\mu$ and $\nu$ satisfies%
\[
\left\Vert DT\right\Vert _{g_{0}}\leq C_{0}.
\]

\end{theorem}

Recall the Wasserstein distance between $\mu$ and $\nu$ is defined as:

\[
W_2^2(\mu, \nu) = \inf_{\gamma \in \Gamma(\mu, \nu)} \int_{M \times M} d^2(x, y) \, d\gamma(x, y)
\]

where:
\begin{itemize}
  \item $d(x, y)$ is the geodesic distance between points $x$ and $y$ on $M$,
  \item $\Gamma(\mu, \nu)$ is the set of all transport plans  $\gamma$ on $M \times M$ with marginals $\mu$ and $\nu$.
\end{itemize}

Regularity for optimal transport is generally an open condition; thus the property is stable in  a broad sense \cite{CF}.  The stability results in \cite{CF} require only smallness on $C^0$ perturbations of the densities, and allow $C^2$ perturbations of the cost.   Quantitative statements on regularity for cost $|x-y|^p$ for $p>1$ in Euclidean domains were obtained in \cite{CGN} and extended in \cite{CF2}. The results in  \cite{CGN} and  \cite{CF2} involve several parameters including the exponent $p$ and the distance between the supports of the measures, and but only require H\"older regularity on the densities.    For parabolic optimal transport equations, results were obtained in \cite{AK} when $c$ is a perturbation of a cost satisfying a weak MTW condition.

Our strategy combines a maximum principle argument with a dichotomy-type bound along the continuity path. Following recent work in \cite{KLM}, we first observe that closeness in Wasserstein space implies small gradients for solutions to the Monge-Amp\`ere equation, provided the measure densities are bounded below by a positive constant. This observation enables the construction of a test function involving the gradient, with constants strong enough to control the evolution—even in the presence of potentially unfavorable MTW terms—so long as we remain along the continuity path.

The dichotomy-type bound refers to inequalities of the form
$$ 10^{541}x \leq 1 + x^{7}$$which ensure that for $x>0$ either $x$  is very small or significantly large (e.g., $x>100$). This kind of bound was previously used in \cite{W_log} to show that highly concentrated Gaussians supported in small regions map smoothly to nearby Gaussians, even in the absence of the MTW condition.

In contrast to the results in  \cite{CGN}, \cite{CF2} and \cite{AK}  our setting of compact manifolds without boundary is much simplified by the fact that we don't need to deal with boundary conditions and can operate directly at the interior point where a maximum is obtained.  On the other hand, our cost function may be quite far from a cost satisfying an MTW condition.    The results in  \cite{CGN}, \cite{CF2},   and  \cite{CF} are obtained using delicate Caffarelli-style level set approximation arguments, and thus the proofs are by nature somewhat different than ours.  The results offered below use a maximum principle, and thus necessarily will require a $C^2$ bound on the log densities; at the same time this also allows us to somewhat painlessly provide clean universal bounds in the statements.  

In \cite{KW} and \cite{JK2} (dealing with optimal transport on boundaries of regions in Euclidean space), a ``stay close-by'' estimate is proved in order keep the optimal transport plan within a region where regularity properties hold.   We argue here, using the recent stability results in  \cite{KLM}, that by choosing Wasserstein distance small enough, one can squeeze the graph of the optimal map into an arbitrarily small neighborhood of the diagonal.  This suggests that one could argue a general stability result (requiring a uniform bound only on the log densities, not their $C^2$ norms)   for manifolds with strictly positive sectional curvature directly using the MTW maximum principle argument.  Note that stability near the diagonal on any Riemannian manifold for $C^0$ perturbations of densities has been shown in \cite[Corollary 2.2]{CF}.

As the $C^2$ bound on the log densities controls the densities themselves, it follows that a small enough $C^2$ bound (i.e $\bar{C}<<1$) can force both conditions in Theorem \ref{main} to hold: thus there is a quantifiable $C^2$ ball in the $C^2$ density norm wherein optimal transport will be smooth (again, cf.  \cite[Corollary 2.2]{CF}).

As we will be absorbing the MTW terms into a general polynomial structure, we do not address the MTW condition in detail here. For a more involved discussion, see the original paper \cite{MTW} or \cite[Chapter 12]{Villani}.

The test function central to our maximum principle argument is inspired by the original work of Korevaar \cite{K}, where the height function of the minimal graph is incorporated to derive estimates for the quasi-linear minimal surface equation. For fully nonlinear equations, this strategy remains useful but typically requires more careful implementation and stronger structural conditions; see, for example, \cite{WY_L} and \cite{BWW}.

\textbf{Acknowledgement} The author is grateful for many conversations with  Farhan Abedin and Jun Kitagawa about this topic.

\section{Preliminaries}

This section is included to set up consistent notation; for more details and
background on the Monge-Amp\`ere equation governing the optimal transport problem, see \cite{Brenier}, \cite{McCann}, \cite{MTW}.  

We will work on a fixed manifold $\left(  M,g_{0}\right)  $ with $c$ the
Riemannian distance-squared function determined by $g_{0}.$

Generally, given a cost function, $c$, on $M\times\bar{M}$ a\ function $u$ is $c$-convex, if
\[
u(x)=\sup_{y\in\bar{M}}u^{\ast}(y)-c(x,y)
\]
for some function $u^{\ast}$ on $\bar{M}.$ \ On a compact manifold with a cost
function that is semi-concave, it follows that a c-convex function will be
semi-convex. For our purposes, this means that in a given set of coordinates, there will be a (negative) lower bound on the Hessian of $u$.  \ \ The distance-squared function on a manifold, being the
infimum of smooth functions, is semi-concave. \ 

We also require the cost to satisfy the (A1) and (A2) conditions as  in \cite{MTW}: The (A1) condition is that given any value $V\in T_{x}^{\ast}M$ there will be
at most one $y\in\bar{M}$ such that
\[
Dc(x,y)=V.
\]
This allows us to implicitly define a transport map $T(x)$ from a $c$-convex
function $u$ via
\begin{equation}
Du(x)+Dc\left(  x,T(x)\right)  =0 \label{a1}%
\end{equation}
whenever $Du(x)$ lies in the range of $y\rightarrow$ $-Dc(x,y).$ \ The (A2)
condition is that
\[
\det(D\bar{D}c)\neq0
\]
(away from a possible cut-locus.) \ Differentiating (\ref{a1}) we get
\[
D^{2}u+D_{x}^{2}c(x,T(x))+D_{x\bar{x}}^{2}c(x,T(x))\cdot DT=0
\]
which is written
\begin{equation}
w_{ij}:=u_{ij}+c_{ij}=-c_{i\bar{s}}T_{j}^{\bar{s}}.\label{asab}%
\end{equation}

Taking determinants, we get
\[
\det\left(  w_{ij}\right)  =\det(-c_{i \bar{s}})\det DT.
\]
From the  Brenier-McCann theory, whenever a $c$-convex function
generates a map that satisfies the Jacobian equation for mass transfer, this
map will be a solution of the optimal transport problem. \ Thus the optimal
transport equation is given by (after taking logs)\
\begin{equation}
\log\det\left(  w_{ij}\right)  =\log\det(-c_{i\bar{s}})+f(x)-g(T(x))
\label{MAE}%
\end{equation}
where
\begin{align*}
\mu &  =e^{f(x)}dx\\
\nu &  =e^{g(T(x))}d\bar{x}.
\end{align*}
\begin{remark}
It should be noted that while working on manifolds, the equation (\ref{MAE}) holds good
given any choice of coordinate systems near $x$ and $T(x)$, even though the
values of the functions $f$ and $g$ depend on the coordinate system itself, as
do $\det(-c_{i\bar{s}})$ and $\det(w_{ij})$.  The equation itself is invariant under change of
coordinate system. The solution $u$ however, is itself well-defined independent of any coordinate system, at least up to an additive constant. 
\end{remark}
Note that when%
\[
c(x, \bar{x}) =\frac{d_{g_{0}}^{2}(x,\bar{x})}{2}%
\]
we have, taking normal coordinates at $\left(  x,x\right)  $
\[
c_{i\bar{s}}=-\delta_{is}.%
\]
It follows that in a neighborhood of the diagonal on a compact manifold, the matrix $-c_{i\bar{s}}$ will be uniformly
positive and the matrix $c_{ij}$ will be uniformly bounded, as will all sets of
mixed or unmixed derivatives $c_{ij\bar{s}k}$ etc. \ All of our computations
will lie in these neighborhoods, in particular, the standard conditions (A1)
and (A2) are immediate in a small enough neighborhood of the diagonal wherein
we will remain. \

Note that although $M$ and $\bar{M}$ are the same manifold, we use different notations to distinguish the source from the target.

While (\ref{MAE}) is geometric,  in our setup for the continuity method we are going to also make use of a
different Riemannian metric $w,$ which depends on the solution $u$.  Given a coordinate system, define as in  (\ref{asab}) above
\[
w(\partial_{i},\partial_{j})(x)=w_{ij}(x)=u_{ij}(x)+c_{ij}(x,T(x)).  
\]

While at first this appears like a second derivative computation in coordinate system,  it is indeed a metric.  This can be seen by observing that it is the induced metric on the graphical submanifold with respect to the Kim-McCann metric \cite{KM},  or by noting that it is the Hessian of the function $u\left(  \cdot\right) +c(\cdot,T(x))$  - whose first derivative vanishes by assumption (\ref{a1}), so the Hessian indeed defines a genuine metric.

\section{Overview of Proof}

The geometric quantity we would like to control is
\begin{equation}
\Lambda(x)=\max_{\left\vert e\right\vert _{g_{0}}=1,e\in T_{x}M}w(e,e).
\label{lambdadef}%
\end{equation}
Without loss of generality, we may perform a one-time scaling on $\left(
M,g_{0}\right)  $ so that 1) The injectivity radius satisfies
\begin{equation}
\iota>2 \label{mep}%
\end{equation} and 2) The metric on normal charts of radius $2$ are equivalent to the Euclidean
metric, that is, for%
\[
\xi=\xi^{i}\partial_{i}%
\]
written in normal coordinates, we have
\begin{equation}
0.9\left\Vert \xi\right\Vert _{g_{0}}^{2}\leq\sum_{i=1}^{n}\left(  \xi
^{i}\right)  ^{2}\leq1.1\left\Vert \xi\right\Vert _{g_{0}}^{2} \label{meq}.%
\end{equation}
\ 

\ It follows that if we can control
\begin{equation}
|w|(x)=\max_{i}w(\partial_{i},\partial_{i}) \label{wabdef}%
\end{equation}
on each ball, we can control the global maximum of $\Lambda.$ \ This in turn
will give local bounds in each chart on $D^{2}u.$ \ 

Continuity is an essential component of the estimate itself. \ The proof plays out in
the following steps. \ \ 

\textbf{Step 1.} Set up a continuity path of measures for $ t\in [0,1]$:%
\[
\rho(t)=(1-t)\mu+t\nu
\]
and observe that the problem of mapping $\mu$ to $\rho(0)=\mu$ is already
solved by $u=0.$ \ By standard elliptic PDE and funtional analysis, there will
be an interval $[0,T)$ upon which the Monge-Ampere equation has a solution.

\textbf{Step 2.} Show that if $W^{2}(\mu,\nu)<\delta_{1}$ $\left(
C_{0},M\right)  $ then for some $\delta_{0}(\delta_{1})$, the solution
$u^{(t)}$ will satisfy
\[
\left\vert \nabla u^{(t)}\right\vert \leq\delta_{0}%
\]
for each $t$ along the path, and this bound can be made arbitrarily small. \ 

\textbf{Step 3.} Given $\delta_{0}$ construct a local test function $\eta$
(involving $\delta_{0}$), a function $\nu=\left\vert w\right\vert ^{c_{n}}$
(the constant $c_{n}$ depends only on $n)$ and show there is a value $C_{3}$
(depending on the geometry of the distance-squared cost for $g_{0}$) such
that if
\begin{equation}
\max_{B_{1}}\eta v\geq C_{3} \label{mc3}%
\end{equation}
then at this point we also have
\begin{equation}
\frac{1}{\delta_{0}^{2}}\left\vert w\right\vert \leq C_{1}+C_{2}\left\vert
w\right\vert ^{3n-1} \label{kd}%
\end{equation}
for $C_{1}$ and $C_{2}$ depending on geometry of the distance-squared cost and
also the log densities of $\mu$ and $\nu,$ but not $\delta_0$. 

\textbf{Step 4.} Given the setup, we now choose $\delta_{0}$ small enough so
that the bound (\ref{kd}) on $\left\vert w\right\vert $ splits possible values
of $\left\vert w\right\vert $ two regions; $[0,a]$ and $[b,\infty),$ with
$a<C_{3}<b.$ \ Then we run a contradiction argument: $\ $The value of
$\left\vert w\right\vert $ can never reach $b.$ \ If it does, it must have
exceeded $C_{3}$ before reaching $b,$ violating (\ref{kd}).
Now by Krylov-Evans and standard elliptic theory, the solution and its bounds continue all the way to $t=1$.  

\section{Step 1:\ Setting up the continuity path.}

We assume that
\[
W^{2}(\mu,\nu)<\delta.
\]
Consider the path of measures
\begin{equation}
\rho(t)=(1-t)\mu+t\nu. \label{rho of t}%
\end{equation}

Let
\[
\mathcal{B}_{1}:=\left\{  u\in C^{2,\alpha}(M)\mid\int ud\mu=0\right\}
\]
and $\mathcal{U\subset B}_{1}$ be the open set
\[
\mathcal{U}:=\left\{  u\in\mathcal{B}_{1}\mid\left\{
\begin{array}
[c]{c}%
u\text{ is strictly c-convex and }\\
T(u)\text{ is a diffeomorphism}%
\end{array}
\right\}  \right\}
\]
and
\[
\mathcal{B}_{2}:=\left\{  h\in C^{\alpha}(M)\mid\log\int e^{h}d\mu=0\right\}
\]
and define%
\[
F:\mathcal{B}_{1}\times\lbrack0,1]\rightarrow\mathcal{B}_{2}%
\]
via the local expression \
\[
F(u,t):=\log\det\left(  u_{ij}+c_{ij}(x,T(x))\right)  -\log\det(-c_{is}%
(x,T(x)))-f(x)+g(T(x),t)
\]
for $T(x)$ determined from $u$ via (\ref{a1}) and where $g(T(x),t)$ is the appropriate target
density corresponding to the measure (\ref{rho of t}) and $f(x)$ is the
density corresponding to $\mu$ in a given coordinate system.

First check that $F$ lands in $\mathcal{B}_{2}:$ \ Given any c-convex $u$, let
$h=F(u,t),$ then
\begin{align*}
e^{h}d\mu &  =\det DTe^{-f}e^{g(T\left(  x,t\right)  }e^{f}dx\\
&  =\det DTe^{g(T\left(  x,t\right) ) }dx
\end{align*}
which is the pullback of a probability measure on the image via a
diffeomorphism. Thus $F$ lands in $\mathcal{B}_{2}.$

We would like to show that for every $t\in\lbrack0,1]$ there is a
$u^{(t)} \in\mathcal{B}_{1}$ such that
\[
F(u^{(t)},t)=0.
\]

First observe that the problem of mapping $\mu$ to $\rho(0)=\mu$ is already
solved by $u=0:$ \ Taking any product of normal systems at a point, it is easy
to check that along the diagonal $(x,x)$
\begin{align*}
c_{ij}  &  =\delta_{ij}\\
-c_{i\bar{s}}  &  =\delta_{i\bar{s}}%
\end{align*}
and the map $T(x)$ constructed from $u=0$ $\ $will be $T(x)=x.$ \ Thus
$F(0,0)=0.$

Now we would like to use continuity as in \cite[Theorem 17.6]{GT}:\ We
claim that $dF$ is invertible on the first factor. \ \ \ Note that the tangent
space to $\mathcal{B}_{1}$ at $u$ is given by
\[
T_{u}\mathcal{B}_{1}=\left\{  z\in C^{2,\alpha}\mid\int zd\mu=0\right\}
\]
and%
\[
T_{F(u,t)}\mathcal{B}_{2}=\left\{  h\in C^{\alpha}\mid\int he^{F(u,t)}%
d\mu=0\right\}
\]
in particular, if at any $t$ we have $F(u,t)=0$, then%
\[
T_{0}\mathcal{B}_{2}=\left\{  h\in C^{\alpha}\mid\int hd\mu=0\right\}.
\]
To use Lax-Milgram, we will set up an operator which behaves well with respect
to the metric $w.$ Define%
\[
\dot{H}^{1}(d\mu)=\left\{  z\mid\int zd\mu=0\text{ and }\int\left\vert
\nabla_{w}z\right\vert _{w}^{2}d\mu<\infty\right\}
\]
which has a nice norm making use of $w$ as a Riemannian metric
\[
\left\Vert z\right\Vert ^{2}:=\int\left\vert \nabla_{w}z\right\vert _{w}%
^{2}d\mu.
\]
Then define
\begin{align*}
a  &  :\dot{H}^{1}(d\mu)\times\dot{H}^{1}(d\mu)\\
a\left(  v,u\right)   &  =-\int w\left(  \nabla_{w}v,\nabla_{w}u\right)
d\mu\\
&  =-\int w\left(  \nabla_{w}v,\nabla_{w}u\right)  e^{\beta}\sqrt{\det w_{ij}%
}dx
\end{align*}
where
\[
\beta=f-\frac{1}{2}\log\det w_{ij}.
\]
Note that integrating by parts gives (divergence with respect to $w)$
\begin{align*}
a\left(  v,u\right)   &  =\int ve^{-\beta}\left(  \operatorname{div}e^{\beta
}\nabla_{w}u\right)  e^{\beta}dVol_{w}\\
&  =\int vLue^{\beta}dVol_{w}%
\end{align*}
for
\begin{align*}
Lu  &  =e^{-\beta}\left(  \operatorname{div}e^{\beta}\nabla_{w}u\right) \\
&  =\Delta_{w}u+\nabla\beta\cdot\nabla u
\end{align*}
Note that at a solution we have%
\begin{align*}
0  &  =\log\det\left(  u_{ij}+c_{ij}(x,T(x))\right)  -\log\det(-c_{is}%
(x,T(x)))-f(x)+g(T(x),t)\\
&  =-\beta+\frac{1}{2}\log\det\left(  u_{ij}+c_{ij}(x,T(x))\right)  -\log
\det(-c_{is}(x,T(x)))+g(T(x),t)
\end{align*}
so we have
\[
\beta=\frac{1}{2}\log\det\left(  u_{ij}+c_{ij}(x,T(x))\right)  -\log
\det(-c_{is}(x,T(x)))+g(T(x),t).
\]
Compute the linearization of $F$
\[
dF(u,t)(z):=w^{ij}\left(  z_{ij}+c_{ij\bar{s}}\left(  -c^{sk}\right)
z_{k}\right)  -c^{\bar{s}i}c_{i\bar{s}\bar{p}}c^{pk}z_{k}-g_{s}c^{sk}z_{k}.
\]
A\ standard computation gives
\[
dF(u,t)(z)=Lz, 
\]
so our goal is to argue that $L$ is invertible, that is for any $h\in$%
\ $T_{0}\mathcal{B}_{2}$ there is $z\in T_{u}\mathcal{B}_{1}$ such that
$Lz=h.$ \ \ 

It is immediate from the definition of $a$ and $\left\Vert \cdot\right\Vert $
that $a$ is coercive on $\dot{H}^{1}(d\mu).$ \ Given $h\in C^{\alpha}$ (which
is $L^{2}$ on a compact manifold) we define
\[
l_{h}\in\left(  \dot{H}^{1}(d\mu)\right)  ^{\ast}%
\]
via
\[
l_{h}(z)=\int zhd\mu.
\]
We check that $l_{h}$ is bounded:
\[
l_{h}(z)\leq\left(  \int h^{2}d\mu\right)  ^{1/2}\left(  \int z^{2}%
d\mu\right)  ^{1/2}\leq C_{h}\left\Vert z\right\Vert
\]
where we have implicity used that a uniformly positive smooth measure on a
compact manifold will enjoy a Poincar\'{e} inequality. \ \ Now by Lax-Milgram,
there exists a unique weak solution $z$ satisfying
\[
-\int w\left(  \nabla_{w}z,\nabla_{w}\eta\right)  d\mu=\int\eta hd\mu.
\]
for all $\eta\in\dot{H}^{1}.$ \ Now at a fixed solution the operator $L$ is a
uniformly elliptic divergence type operator, we may apply standard elliptic
regularity theory to conclude that $v\in C^{2,\alpha}$ satisfies
\[
Lz=h.
\]
The conditions in \cite[Theorem 17.6]{GT} are satisfied and we proceed along
the continuity path.

\section{Step 2: Gradient can be squeezed by Wasserstein}

Recall the result from \cite[Theorem 1.2 ]{KLM}, adapted here for our purposes:

\begin{theorem} (\cite{KLM})
Let $M$ be a compact smooth connected Riemannian manifold without boundary,
endowed with the quadratic cost $c=\frac{d^{2}(x,y)}{2}$ and let $\rho$ be a
probability measure absolutely continuous with respect to the Riemamnnian
volume on $M$ with density bounded from above and below by positive constants.
Then there exists $C>0$ such that for any $\mu,\nu\in P(M)$%
\[
\int\left\vert \phi_{\mu}-\phi_{\nu}\right\vert ^{2}d\rho\leq CW_{1}(\mu
,\nu)^{1/2}%
\]
where $\phi_{\mu},\phi_{\nu}$ are Kantorovich potentials pairing $\rho$ with
$\mu,\nu$ respectively and satisfying $\int\phi d\rho=0.$
\end{theorem}

In our case, we may take $\mu$ to take the place of $\rho$ in which case
$\phi_{\mu}=0$ and we get
\begin{equation}
\int\left\vert u\right\vert ^{2}d\mu  \leq CW_{1}(\mu,\nu)^{1/2}  \leq CW_{2}(\mu,\nu)^{1/4}  \leq C\delta^{1/8}. \label{preb}%
\end{equation}
Now $u$ is a semiconvex function: that is, on a given chart, we will
have
\[
D^{2}u\geq-A
\]
for some uniform $A.$ We next obtain  a lower bound on the $L^{2}$ norm of $u$ from its  interior
gradient maximum:

\begin{proposition}
\label{bbbb}Suppose that $u$ is a semi-convex function on $M$. \ Let
\[
b=\max\left\vert \nabla u\right\vert .
\]
Then there is a value $c(A,g,\mu)>0$ and a value $C(A,g,\mu)$ such that if
\[
\int\left\vert u\right\vert ^{2}d\mu\leq c(A,g,\mu)
\]
then
\[
b^{n+4}\leq C(A,g,\mu)\left\Vert u\right\Vert _{L^{2}(d\mu)}^{2}.
\]

\end{proposition}

\begin{proof}
Take normal coordinates at the point where $\left\vert \nabla u\right\vert $
is maximized, such that
\[
\frac{\nabla u}{\left\vert \nabla u\right\vert }=\partial_{1}.%
\]
There is a ball of radius $1$ such that on $B_{1}$ (recalling (\ref{meq})) that
\begin{equation}
\left\vert Du\right\vert \leq1.1b \label{meq2}.%
\end{equation}
Also, using
\[
u_{11}\geq-A
\]
we have
\[
u_{1}(x_{1},0,...,0)\geq b-Ax_{1}\text{ for }x_{1}>0.
\]
Now first we consider the case that
\[
b<A.
\]
Then
\begin{align*}
u(\frac{b}{2A},...,0)-u(0)  &  \geq\int_{0}^{\frac{b}{2A}}\left(
b-Ax_{1}\right)  dx_{1}\\
&  =\frac{3b^{2}}{8A}.
\end{align*}
Now let
\begin{align*}
Q_{1}  &  =B_{\frac{b}{11A}}(0)\\
Q_{2}  &  =B_{\frac{b}{11A}}(\frac{b}{2A},...,0))
\end{align*}
Then by (\ref{meq2})
\begin{align*}
\max_{x\in Q_{1}}\left(  u(x)-u(0)\right)   &  \leq1.1b\frac{b}{11A}\\
\min_{x\in Q_{1}}\left(  u(x)-u(\frac{b}{2A},...,0)\right)   &  \geq
-1.1b\frac{b}{11A}.%
\end{align*}
Thus
\begin{align*}
\max_{x\in Q_{1}}u(x)  &  \leq\frac{b^{2}}{10A}+u(0)\\
\min_{x\in Q_{2}}u(x)  &  \geq-\frac{b^{2}}{10A}+u(\frac{b}{2A},...,0)\\
&  \geq-\frac{b^{2}}{10A}+\frac{3b^{2}}{8A}+u(0)\\
&  \geq\frac{7b^{2}}{40A}+\max_{x\in Q_{1}}u(x)
\end{align*}
which means on at least one of the two balls, we have
\[
\left\vert u\right\vert \geq\frac{7b^{2}}{80A}.
\]
It follows that
\begin{align*}
\left\Vert u\right\Vert _{L^{2}(d\mu)}^{2}  &  \geq\min\mu(x)c_{n}\left(
\frac{b}{11A}\right)  ^{n}\left(  \frac{7b^{2}}{80A}\right)  ^{2}\\
&  \geq c(A,\mu)b^{n+4}.
\end{align*}
Next we claim that $b\geq A$ can be excluded by choosing $c(A,g,\mu)$ small
enough. \ This follows a nearly identical argument, replacing
\begin{align*}
&  u(\frac{b}{2A},...,0)\text{ with }u(\frac{1}{2},...,0)\\
&  B_{\frac{b}{11A}}(0)\text{ with }B_{\frac{1}{11}}(0)
\end{align*}
giving%
\[
u(\frac{1}{2},...,0)-u(0)\geq\frac{3b}{8}%
\]
etc, and showing that
\[
\left\Vert u\right\Vert _{L^{2}(d\mu)}^{2}\geq\min\mu(x)c_{n}\left(  \frac
{1}{11}\right)  ^{n}\left(  \frac{7}{80}b\right)  ^{2}.
\]

\end{proof}

We note here that as one can construct a candidate plan pairing mass between $\mu$ and convex combination of $\mu$ and $\nu$ using the partial transport of the optimal plan, the Wasserstein distance between $\mu$ and convex combinations of  $\mu$ and $\nu$ is bounded above by the Wasserstein distance between $\mu$ and $\nu.$  Thus any estimates obtained from $\delta$ are good along the full path.  

\section{Step 3: Dichotomy bound}

In this section we operate assuming that we have a strong gradient bound. 

\begin{proposition}
\label{chaz}Suppose that $u$ is a solution to (\ref{MAE}) such that the graph
lies in a compact set $K$ away from the cut locus, and suppose that
\[
\left\vert \nabla u\right\vert \leq\delta_{0}.
\]
There exists $C_{1},C_{2}$ which depend on $M,g_{0},\left\Vert f\right\Vert
_{C^{2}},\left\Vert g\right\Vert _{C^{2}},\left\Vert D^{4}c\right\Vert _{K}$
but not on $\delta_{0},$ such that if
\[
\max\Lambda\geq C_{3}:=4n\left(  2.3\right)  ^{11n}\max_{K}D^{2}c
\]
(recall (\ref{lambdadef}))  then
\begin{equation}
\frac{1}{\delta_{0}^{2}}\Lambda\leq C_{1}+C_{2} \Lambda
^{3n-1}. \label{step3}%
\end{equation}

\end{proposition}

\begin{proof}
Starting with $\Lambda\geq C_{3}$, there is a point $x_{0}$ and unit vector $e \in T_{x_0}M$  such that
\[
w(e,e)=\Lambda\geq C_{3}.
\]
Take normal coordinates for $g_{0}$, centered at $x_{0}$ and then (recall
(\ref{mep})) consider the following test function
\[
\eta=\left(  1+\frac{1}{\delta_{0}^{2}}|Du|^{2}-2|x|^{2}\right)  ^{+}%
\]
which vanishes outside of $B_{\frac{5}{4}}.$ (To be clear $|Du|^{2}$ is a
\textquotedblleft non-geometric" coordinate quantity, but this is OK).\ Now recall (\ref{wabdef}) and define
\[
\bar{V}:=\exp\left(  c_{n}\ln\left\vert w\right\vert \right)
\]
and go to the point in $B_{\frac{5}{4}}$, which we label $x_{\max},$ where $\eta\bar{V}$ is maximized. \ At this point, diagonalize (via a rigid Euclidean rotation) $w$ so that the maximum eigenvalue
is in the $\partial_{1}$ coordinate direction. We will be taking $c_{n}%
=\frac{1}{11n}.$ \ Since $\eta(0)\geq1$ and $\eta(x_{\max})\leq1+\left(
1.1\right)  ^{2}$ we conclude that
\[
\bar{V}(x_{\max})\geq\frac{1}{2.21}\bar{V}(0)
\]
that is
\begin{align}
w_{11}(x_{\max})  &  =\left\vert w\right\vert (x_{\max})\nonumber\\
&  \geq\left(  \frac{1}{2.21}\right)  ^{11n}\left\vert w\right\vert
(0)\geq\left(  \frac{1}{2.21}\right)  ^{11n}C_{3}=:\tilde{C}_{3}. \label{cmui}%
\end{align}
Now define the function \
\[
V:=\exp\left(  c_{n}\ln w_{11}\right)  .
\]
Since $\eta V\leq\eta\bar{V}$ and $\eta V=\eta\bar{V}$ at $x_{\max}$ it must
also be the case that $\eta V$ has a maximum at $x_{\max}$. We will
show\ below that $V$ satisfies the following when $w_{11}\geq\tilde{C}_{3}:$
\begin{equation}
w^{ij}V_{ij}-\frac{2w^{ij}V_{i}V_{j}}{V}\geq-C_{4}\left(  1+\left\vert
w\right\vert ^{3n-1}\right)  V \label{jacobi}%
\end{equation}
for some $C_{4}(n,D^{4}c,M,D^{2}f,D^{2}g)\ $.

Now at $x_{\max}$ where $\eta V$ is at a maximum
\[
\eta_{i}V+\eta V_{i}=0
\]%
\begin{align*}
w^{ij}(\eta V)_{ij}  &  =w^{ij}\eta_{ij}V+w^{ij}V_{ij}\eta+2w^{ij}V_{i}\left(
-\frac{\eta V_{j}}{V}\right) \\
&  =w^{ij}\eta_{ij}V+\eta\left(  w^{ij}V_{ij}-\frac{2w^{ij}V_{i}V_{j}}%
{V}\right)  \leq0.
\end{align*}
Thus assuming (\ref{jacobi}) we have
\[
w^{ij}\eta_{ij}V\leq\eta C_{4}\left(  1+\left\vert w\right\vert ^{3n-1}%
\right)  V
\]
and we conclude that at $x_{\max}$
\begin{equation}
w^{ij}\eta_{ij}\leq C_{4}\left(  1+\left\vert w\right\vert ^{3n-1}\right)  .
\label{cl0}%
\end{equation}
To arrive at the conclusion (\ref{step3}) it will be crucial then to show two
key facts at $x_{\max}$:\ \ First;
\begin{equation}
w^{ij}\eta_{ij}\geq\frac{1}{2}\frac{\left\vert w\right\vert }{\delta_{0}^{2}%
}-C_{1}^{\prime}-C_{2}^{\prime}|w|^{2n-1} \label{cl1}%
\end{equation}
and second, (\ref{jacobi}). \ 

We begin by showing  (\ref{cl1}). \ Compute%
\begin{equation}
w^{ij}\eta_{ij}=\frac{2}{\delta_{0}^{2}}w^{ij}u_{ki}u_{kj}+\frac{2}{\delta
_{0}^{2}}u_{k}w^{ij}u_{ijk}-4w^{ij}. \label{cf1}%
\end{equation}
Now
\begin{align}
w^{ij}u_{ki}u_{kj}  &  =w^{ij}\left(  w_{ki}-c_{ki}\right)  \left(
w_{kj}-c_{kj}\right) \nonumber\\
&  =\sum w_{ii}-2\sum c_{ii}+w^{ij}c_{ki}c_{kj}\nonumber\\
&  \geq w_{11}-2n\max_{K}D^{2}c\geq\frac{1}{2}w_{11}\ \label{cf2}%
\end{align}
where in the last line we use (this is where the $C_{3}$ condition is
required)\
\[
w_{11}(x_{\max})\geq\left(  \frac{1}{2.21}\right)  ^{11n}4n\left(  2.3\right)
^{11n}\max_{K}D^{2}c.
\]
Now for ease of notation we introduce
\[
\zeta(x,\bar{x})=\ln\det(-c_{i\bar{s}}(x,\bar{x}))
\]
so differentiating (\ref{MAE}) we get
\[
w^{ij}w_{ij,k}=f_{k}-g_{\bar{s}}T_{k}^{\bar{s}}+\zeta_{k}+\zeta_{\bar{s}}%
T_{k}^{\bar{s}}.
\]
Now
\[
w_{ij,k}=u_{ijk}+c_{ijk}+c_{ij\bar{s}}T_{k}^{\bar{s}}%
\]
so%
\[
w^{ij}u_{ijk}=w^{ij}w_{ij,k}-w^{ij}\left(  c_{ijk}+c_{ij\bar{s}}T_{k}^{\bar
{s}}\right)
\]
and combining with (\ref{cf1}) and (\ref{cf2})%
\begin{align*}
w^{ij}\eta_{ij}  &  \geq\frac{1}{\delta_{0}^{2}}w_{11}\ -4w^{ii}\\
&  +\frac{2}{\delta_{0}^{2}}u_{k}\left(  f_{k}-g_{\bar{s}}T_{k}^{\bar{s}%
}+\zeta_{k}+\zeta_{\bar{s}}T_{k}^{\bar{s}}-w^{ij}\left(  c_{ijk}+c_{ij\bar{s}%
}T_{k}^{\bar{s}}\right)  \right)  .
\end{align*}
Next we use the assumption $\left\vert u_{k}\right\vert \leq\delta_{0}$.
Since
\[
\left\vert T_{k}^{s}\right\vert =\left\vert \left(  -c^{\bar{s}i}w_{ik}\right)
\right\vert \leq C_{5}w_{11}%
\]
for
\[
C_{5}:=\sup_{K}\left\vert \left(  D\bar{D}c\right)  ^{-1}\right\vert
\]
which exists from the (A2) condition, we have
\begin{align*}
&  \frac{2}{\delta_{0}^{2}}u_{k}\left(  f_{k}-g_{\bar{s}}T_{k}^{\bar{s}}%
+\zeta_{k}+\zeta_{\bar{s}}T_{k}^{\bar{s}}-w^{ij}\left(  c_{ijk}+c_{ij\bar{s}%
}T_{k}^{\bar{s}}\right)  \right) \\
&  \geq-\frac{2\delta_{0}C_{6}}{\delta_{0}^{2}}\left(  1+\left\vert
w\right\vert +\left\vert w^{-1}\right\vert +\left\vert w^{-1}\right\vert
|w|\right) \\
&  \geq-\frac{1}{2\delta_{0}^{2}}w_{11}-\frac{1}{w_{11}}\left[  C_{6}\left(
1+\left\vert w\right\vert +\left\vert w^{-1}\right\vert +\left\vert
w^{-1}\right\vert |w|\right)  \right]  ^{2}%
\end{align*}
for $C_{6}$ depending on $Df,Dg$,$D\zeta,D^{3}c$ and $C_{5}.$\newline Now
since
\[
\det w_{ij}\geq c_{7}>0
\]
for $c_{7}$ depending on $\left\vert f\right\vert $ and $\left\vert
g(\cdot,t)\right\vert $ we have
\[
\left\vert w^{-1}\right\vert \leq\frac{1}{c_{7}}\left(  \left\vert
w\right\vert ^{n-1}\right)
\]
and
\begin{align*}
w^{ij}\eta_{ij}  &  \geq\frac{1}{\delta_{0}^{2}}w_{11}\ -4w^{ii}-\frac
{1}{2\delta_{0}^{2}}w_{11}-\left[  C_{6}\left(  1+\left\vert w\right\vert
+\left\vert w^{-1}\right\vert +\left\vert w^{-1}\right\vert |w|\right)
\right]  ^{2}\\
&  \geq\frac{1}{2\delta_{0}^{2}}w_{11}-C_{8}-C_{9}\left\vert w\right\vert
^{2n-1},%
\end{align*}
for $C_{8}$ and $C_{9}$ determined from previous constants and power interpolation
inequalities and also the fact that $\left\vert w\right\vert
(x_{\max})=w_{11}$. This proves (\ref{cl1}).

Next, let
\[
b=\ln w_{11}%
\]
we will show that for $c_{n}=\frac{1}{11n}$%
\begin{equation}
w^{ij}b_{i}b_{j}\geq-C_{4}^{\prime}\left(  1+\left\vert w\right\vert
^{3n-1}\right)  +c_{n}w^{ij}b_{i}b_{j}. \label{htr}%
\end{equation}
One can check that (\ref{htr}) implies (recall $V=\exp(c_{n}b)$)
\[
w^{ij}V_{ij}-2w^{ij}\frac{V_{i}V_{j}}{V}\geq-c_{n}C_{4}^{\prime}\left(
1+\left\vert w\right\vert ^{3n-1}\right)  V.
\]
To show (\ref{htr}) differentiate
\begin{align*}
b_{i}  &  =\frac{w_{11,i}}{w_{11}}=w^{11}w_{11,i}\\
b_{ij}  &  =w^{11}w_{11,ij}-w^{1a}w^{b1}w_{ab,j}w_{11,i}\\
w^{ij}b_{ij}  &  =w^{ij}\left(  w^{11}w_{11,ij}-w^{1a}w^{b1}w_{ab,j}%
w_{11,i}\right)  .
\end{align*}
Calling
\[
h_{ij,k}^{2}:=w^{kk}w^{ii}w^{jj}w_{ij,k}^{2}.
\]
and using diagonalization of $w_{ij}$, we find that
\begin{equation}
w^{ij}b_{ij}=w^{11}w^{ij}w_{11,ij}-h_{11,i}^{2}. \label{operator-leading-term}%
\end{equation}
We also have
\begin{align*}
w^{ij}b_{i}b_{j}  
&  =h_{11,i}^{2}.%
\end{align*}
Therefore,
\begin{equation}
w^{ij}b_{ij}-c_{n}w^{ij}b_{i}b_{j}=w^{11}w^{ij}w_{11,ij}-(1+c_{n})h_{11,i}%
^{2}. \label{here}%
\end{equation}
We focus on the term $w^{11}w^{ij}w_{11,ij}$. Differentiating \eqref{MAE}
twice we get
\begin{align*}
w^{ij}w_{ij,1}  &  =\zeta_{1}+\zeta_{\bar{s}}T_{1}^{\bar{s}}+f_{1}-g_{s}T_{1}^{\bar{s}}\\
w^{ij}w_{ij,11}-w^{ia}w^{bj}w_{ab,1}w_{ij,1}  &  =\zeta_{11}+\zeta_{1\bar{s}%
}T_{1}^{\bar{s}}+(\zeta_{1\bar{s}}+\zeta_{\bar{s}\bar{p}}T_{1}^{\bar{p}})T_{1}^{\bar{s}}\\
&  +\zeta_{\bar{s}}T_{11}^{\bar{s}}+f_{11}-g_{\bar{s}\bar{p}}T_{1}^{\bar{p}}T_{1}^{\bar{s}}-g_{\bar{s}}   T_{11}^{\bar{s}}\\
&  \geq-C_{10}(1+\left\vert w\right\vert ^{2})+\left(  \zeta_{\bar{s}}%
-g_{\bar{s}}\right)  T_{11}^{\bar{s}}%
\end{align*}
for $C_{10}$ depending on $C_{5},$ $D^{2}f,D^{2}g,$ and $D^{2}\zeta.$
\ Therefore,
\begin{align}
&  w^{ij}w_{11,ij}=w^{ij}(w_{11,ij}-w_{ij,11})+w^{ij}w_{ij,11}\nonumber\\
&  \geq w^{ij}(w_{11,ij}-w_{ij,11})+\sum_{i,j}h^2_{ij,1}\nonumber\\
&  -C_{10}(1+\left\vert w\right\vert ^{2})+\left(  \zeta_{\bar{s}}-g_{\bar{s}%
}\right)  T_{11}^{\bar{s}} .\label{cv2}%
\end{align}
A standard computation (as in \cite{MTW}) yields%
\[
w_{11,ij}-w_{ij,11}\geq-C_{11}(1+\left\vert w\right\vert ^{2})+c_{11,\bar{s}%
}T_{ij}^{\bar{s}}+c_{ij,\bar{s}}T_{11}^{\bar{s}}%
\]
for $C_{11}$ depending on $C_{5},$ $D^{2}f,D^{2}g,$ $D^{2}\zeta.$and $D^{4}c$,
thus%
\begin{equation}
w^{ij}(w_{11,ij}-w_{ij,11})\geq-C_{12}(1+\left\vert w\right\vert
^{n+1})+c_{11,\bar{s}}w^{ij}T_{ij}^{\bar{s}}+c_{ij,\bar{s}}w^{ij}T_{11}^{\bar{s}}
\label{cv4}%
\end{equation}
for $C_{12}$ depending on $c_{7}, C_{11}$ and polynomial interpolation inequalities. \ Now
differentiating
\[
w^{ij}w_{ki,j}=w^{ij}\left(  -c_{k\bar{s}}T_{i}^{\bar{s}}\right)  _{j}=-c_{k\bar
{s}j}w^{ij}T_{i}^{\bar{s}}-c_{k\bar{s}}w^{ij}T_{ij}^{\bar{s}}%
\]
we get
\begin{align}
w^{ij}T_{ij}^{s}  &  =\left(  -c^{k \bar{s}}\right)  w^{ij}w_{ki,j}+\left(  -c^{k \bar{s}}\right)c_{k\bar{s}%
j}w^{ij}T_{i}^{s}\nonumber\\
&  =\left(  -c^{k \bar{s}}\right)  w^{ij}w_{ij,k}+\left(  -c^{ks}\right)
w^{ij}\left(  w_{ki,j}-w_{ij,k}\right)  +\left(  -c^{k \bar{s}}\right)c_{k\bar{s}j}\left(  -c^{sj}\right)
\nonumber\\
&  \geq-C_{13}(1+\left\vert w\right\vert ^{n}). \label{cv5}%
\end{align}
Also
\begin{align*}
T_{11}^{\bar{s}}  &  =\left( - c^{k\bar{s}}w_{k1}\right)  _{1}=-c_{1}^{\bar{s}k}%
w_{k1}-c^{\bar{s}k}w_{k1,1}\\
&  \leq C_{14}\left(  1+\left\vert w\right\vert ^{2}\right)  +\sum_{k}%
C_{5}\sqrt{w_{kk}}w_{11}\frac{w_{k1,1}}{\sqrt{w_{kk}}w_{11}}.
\end{align*}
Hence
\begin{align}
c_{ij,\bar{s}}w^{ij}T_{11}^{\bar{s}}  &  \geq-C_{15}\left\vert w\right\vert
^{n-1}\left(  C_{14}\left(  1+\left\vert w\right\vert ^{2}\right)  +\sum
_{k}C_{5}\sqrt{w_{kk}}w_{11}\frac{w_{k1,1}}{\sqrt{w_{kk}}w_{11}}\right)
\nonumber\\
&  \geq-C_{16}\left\vert w\right\vert ^{n-1}\left(  1+\left\vert w\right\vert
^{2}\right)  -\frac{\left(  C_{15}C_{5}\right)  ^{2}}{\varepsilon_{n}%
}\left\vert w\right\vert ^{2n-2}\left\vert w\right\vert ^{3}-\varepsilon
_{n}h_{k1,1}^{2} \label{cv6}%
\end{align}
for some $\varepsilon_{n}$ to be chosen. \ Combining (\ref{cv2}), (\ref{cv4}),
(\ref{cv5}), (\ref{cv6}), and using reasoning leading to (\ref{cv6}) on the remaining terms\ gives
\begin{equation}
w^{ij}w_{11,ij}\geq-C_{17}\left(  1+\frac{1}{\varepsilon_{n}}\left\vert
w\right\vert ^{2n+1}\right)  -2\varepsilon_{n}h_{k1,1}^{2}+\sum_{i,j}%
h_{ij,1}^{2}. \label{there}%
\end{equation}

Now compute that
\begin{align*}
&  h_{ij,k}^{2}-h_{kj,i}^{2}=w^{kk}w^{ii}w^{jj}(w_{ij,k}^{2}-w_{kj,i}^{2})\\
&  =w^{kk}w^{ii}w^{jj}(w_{ij,k}-w_{kj,i})(w_{ij,k}+w_{kj,i})
\end{align*}
and
\begin{align*}
w_{ij,k}  &  =u_{ijk}+c_{ijk}+c_{ijs}T_{k}^{s}\\
w_{kj,i}  &  =u_{kji}+c_{jki}+c_{jks}T_{i}^{s}%
\end{align*}

So
\begin{align*}
w^{kk}w^{ii}w^{jj}(w_{ij,k}-w_{kj,i})(w_{ij,k}+w_{kj,i}) &  \leq w^{kk}%
w^{ii}w^{jj}C_{18}\left\vert w\right\vert \left(  (w_{ij,k}+w_{kj,i})\right)
\\
&  \leq w^{kk}w^{ii}w^{jj}\left(  \frac{2C_{18}^{2}}{\varepsilon_{n}%
}\left\vert w\right\vert ^{2}+\frac{\varepsilon_{n}}{2}(w_{ij,k}+w_{kj,i}%
)^{2}\right)  \\
&  \leq\left(  \frac{1}{c_{7}}\right)  ^{3}\frac{2C_{18}^{2}}{\varepsilon_{n}%
}\left\vert w\right\vert ^{3n-1}+\varepsilon_{n}\left(  h_{ij,k}^{2}%
+h_{kj,i}^{2}\right) .
\end{align*}
which gives
\[
h_{ij,k}^{2}\leq C_{19}\left(  \varepsilon_{n}\right)  \left\vert w\right\vert
^{3n-1}+\frac{\left(  1+\varepsilon_{n}\right)  }{(1-\varepsilon_{n})}%
h_{kj,i}^{2}%
\]
and we can apply to (\ref{here}) using (\ref{there})
\begin{align}
w^{ij}b_{ij}-c_{n}w^{ij}b_{i}b_{j} &  \geq-C_{17}\left(  1+\frac
{1}{\varepsilon_{n}}\left\vert w\right\vert ^{2n+1}\right)  -2\varepsilon
_{n}h_{k1,1}^{2}+\sum_{i,j}h_{ij,1}^{2}\label{se4}\\
&  -\sum_{i}(1+c_{n})\left(  C_{19}\left(  \varepsilon_{n}\right)  \left\vert
w\right\vert ^{3n-1}+\frac{\left(  1+\varepsilon_{n}\right)  }{(1-\varepsilon
_{n})}h_{i1,1}^{2}\right)  \nonumber\\
&  \geq-C_{20}(\varepsilon_{n},c_{n})(1+\left\vert w\right\vert ^{3n-1}%
)\ +\sum_{i,j}h_{ij,1}^{2}\nonumber\\
&  -\sum_{k}\left(  (1+c_{n})\frac{\left(  1+\varepsilon_{n}\right)
}{(1-\varepsilon_{n})}+2\varepsilon_{n}\right)  h_{k1,1}^{2}.\nonumber
\end{align}

Now observe that%

\begin{align}
\sum_{i,j}h_{ij,1}^{2}-\sum_{k}\left(  1+\frac{1}{2n}\right)  h_{k1,1}^{2}  &
=\sum_{k}-\frac{1}{2n}h_{k1,1}^{2}+\sum_{k,i>1}h_{ki,1}^{2}\label{se7}\\
&  \geq(1-\frac{1}{2n})\sum_{k>1}h_{k1,1}^{2}\nonumber\\
&  -\frac{1}{2n}h_{11,1}^{2}+h_{22,1}^{2}+....+h_{nn,1}^{2}.\nonumber
\end{align}
Now
\[
h_{11,1}+h_{22,1}+...+h_{nn,1}=\sqrt{w^{11}}w^{kk}w_{kk,1}=\sqrt{w^{11}%
}\left[  \zeta_{1}+\zeta_{\bar{s}}T_{1}^{s}+f_{1}-g_{\bar{s}}T_{1}^{\bar{s}%
}\right]
\]
so
\[
h_{11,1}=-\left(  h_{22,1}+...+h_{nn,1}\right)  +\sqrt{w^{11}}\left[
\zeta_{1}+\zeta_{\bar{s}}T_{1}^{s}+f_{1}-g_{\bar{s}}T_{1}^{\bar{s}}\right]
\]
so
\begin{align*}
h_{11,1}^{2}  &  \leq2\left(  h_{22,1}+...+h_{nn,1}\right)  ^{2}+\frac
{2C_{21}}{w_{11}}\left[  1+\left\vert w\right\vert ^{2}\right]  ^{2}\\
&  \leq2n\left(  h_{22,1}^{2}+...+h_{nn,1}^{2}\right)  +C_{21}(1+\left\vert
w\right\vert ^{3})
\end{align*}
plugging back into (\ref{se7})
\begin{align*}
\sum_{i,j}h_{ij,1}^{2}-\sum_{k}\left(  1+\frac{1}{2n}\right)  h_{k1,1}^{2} &
\geq-\frac{1}{2n}\left(  2n\left(  h_{22,1}^{2}+...+h_{nn,1}^{2}\right)
+C+C\left\vert w\right\vert ^{3}\right)  \\
&  +h_{22,1}^{2}+...+h_{nn,1}^{2}\\
&  \geq-\frac{1}{2n}C_{21}(1+\left\vert w\right\vert ^{3})
\end{align*}
So we choose $\varepsilon_{n}$ and $c_{n}$ small enough so that ($\varepsilon
_{n}=c_{n}=\frac{1}{11n}$ will work)
\[
\left(  (1+c_{n})\frac{\left(  1+\varepsilon_{n}\right)  }{(1-\varepsilon
_{n})}+2\varepsilon_{n}\right)  \leq1+\frac{1}{2n}%
\]
and we have from (\ref{se4})
\[
w^{ij}b_{ij}-c_{n}w^{ij}b_{i}b_{j}\geq-C_{20}(\varepsilon_{n},c_{n}%
)(1+\left\vert w\right\vert ^{3n-1})-\frac{1}{2n}C_{21}(1+\left\vert
w\right\vert ^{3}).
\]
Proving (\ref{htr}). \ 

Recapping and obtaining (\ref{step3}): \ We assumed that $\left\vert
w\right\vert (x_{0})\geq C_{3}$ and conclude that at a point nearby, we have
(\ref{cmui})
\[
w_{11}(x_{\max})\geq\left(  \frac{1}{2.21}\right)  ^{11n}\left\vert
w\right\vert (x_{0})
\]
which satisfies at $x_{\max}$ (\ref{cl1}), (\ref{cl0})
\begin{equation}
\frac{1}{2}\frac{w_{11}}{\delta_{0}^{2}}\leq C_{4}(1+w_{11}^{3n-1}).
\label{brig}%
\end{equation}
Plugging
\[
w_{11}(x_{\max})=\text{ }\sigma\Lambda(x_{0})
\]
for $\sigma\in\lbrack\left(  \frac{1}{2.21}\right)  ^{11n},1]$ into
(\ref{brig}) yields (\ref{step3}) with $C_{1}=\left(  \frac{1}{2.21}\right)
^{11n}C_{4}$ and $C_{2}=C_{4}.$

\end{proof}

\section{Step 4: conclusion}

To prove Theorem \ref{main}, choose any two measures satisfying the condition
\[
\left\Vert \log\frac{d\mu}{dVol}\right\Vert _{C^{2}},\left\Vert \log\frac
{d\nu}{dVol}\right\Vert _{C^{2}}\leq\bar{C}.
\]
We claim that the Hessian of the solution to the scalar optimal transport equation for these measures is uniformly bounded by $C_3+1$,
provided that
\begin{equation}
W^{2}(\mu,\nu)<\delta\label{dde}%
\end{equation}
for some $\delta$ to be chosen. \ \ By Proposition \ref{chaz}, we can choose
$\delta_{0}$\ small enough so that the bound on $\Lambda$ by
\begin{equation}
\frac{1}{\delta_{0}^{2}}\Lambda\leq C_{1}+C_{2}\Lambda^{3n-1} \label{ddd}%
\end{equation}
splits the possible values of $\Lambda$ into two regions $\left[  0,a\right]
$ and $[b,\infty)$ and with
\[
a<C_{3}< C_3+1< b.
\]
Now chose $\delta$ according to (\ref{preb}) and Proposition \ref{bbbb} so
that
\[
\left\vert \nabla u\right\vert \leq\delta_{0}%
\]
for any $u$ solving an optimal transportation equation between measures
satisfying (\ref{dde}). \ Start the continuity method, mapping $\mu$ to $\mu$
at time $0$ and moving towards $t=1.$ \ We claim that $\Lambda<b$ for every
time along this path. \ If not, then there must be a time where the max value
of $\Lambda$ is exactly $C_{3}.$ \ At this point we apply (\ref{ddd}) and get
a contradiction. \ 

Note that this bound is also what allows us to continue along the full path.  Once the $C^2$ bounds are shown, we can use Krylov-Evans theory which applies to the concave Monge-Amp\`ere equation.    The $C^{2,\alpha}$ estimates guarantee that the set of $t \in [0,1]$ on which a solution exists is closed.  

\bibliographystyle{plain}
\bibliography{ot}

\end{document}